\documentclass[12pt, reqno]{amsart}
\usepackage{amsmath, amsthm, amscd, amsfonts, amssymb, graphicx, color}
\usepackage[bookmarksnumbered,plainpages]{hyperref}
\usepackage[all]{xy}

\newtheorem{theorem}{Theorem}[section]
\newtheorem{lemma}[theorem]{Lemma}
\newtheorem{proposition}[theorem]{Proposition}
\newtheorem{corollary}[theorem]{Corollary}
\usepackage{enumerate}
\theoremstyle{definition}

\newtheorem{example}[theorem]{Example}

\newtheorem{Theorem}{\quad Theorem}[section]

\newtheorem{remark}[Theorem]{\quad Remark}
\numberwithin{equation}{section}
\input{mathrsfs.sty}

\begin{document}
\title[The operator--valued parallelism]{The operator--valued parallelism and norm-parallelism in matrices}
\author[ M. Mohammadi Gohari and M. Amyari]{M. Mohammadi Gohari and M. Amyari$^*$}
\address{Department of Mathematics, Mashhad Branch,
 Islamic Azad University, Mashhad, Iran.}
  \email{mahdi.mohammadigohari@gmail.com}
\email{maryam\_amyari@yahoo.com  and  amyari@mshdiau.ac.ir}
\subjclass[2010]{Primary 46L08, Secondary 47A30, 47B10, 47B47}
\keywords{Hilbert $C^*$-module, orthonormal basis, parallelism, Schatten p-norm.\\
*Corresponding author}
\begin{abstract}
Let $\mathcal{H}$ be a Hilbert space, and let $K(\mathcal{H})$ be the $C^*$-algebra of compact operators
on $\mathcal{H}$. In this paper,  we present some characterizations of the norm-parallelism for elements
 of a Hilbert $K(\mathcal{H})$-module by employing the Birkhoff--James orthogonality.
 Among other things, we present a characterization of transitive relation of the norm-parallelism for elements in a certain Hilbert $K(\mathcal{H})$-module.
 We also give some characterizations of the Schatten $p$-norms and the operator norm-parallelism for matrices.
\end{abstract}
\maketitle
\section{Introduction and preliminaries}
Throughout the  paper, let $(\mathcal{H}, [\cdot, \cdot ])$  be a Hilbert space. Let $B(\mathcal{H})$ and $K(\mathcal{H})$
 denote the $C^*$-algebras of all bounded linear operators and compact linear operators on $\mathcal{H}$, respectively.
 The numerical range of an operator $T\in B(\mathcal{H})$ is defined by $W(T)=\{[ T\xi,\xi]:~~~\xi\in\mathcal{H},~~~\|\xi\|=1\}$.
 Using the Gelfand--Naimark theorem, we identify any $C^*$-algebra $\mathcal{A}$ as a $C^*$-subalgebra of $B(\mathcal{H})$ for some Hilbert space $\mathcal{H}$. An (left) inner product $\mathcal{A}$-module is a left $\mathcal{A}$-module $\mathcal{E}$
  equipped with an $\mathcal{A}$-valued inner product $\langle \cdot, \cdot \rangle,$ which is linear
  in the first variable and conjugate linear in the second and satisfies  (i) $\langle x,x \rangle\geq 0$ and $\langle x,x \rangle=0$ if and only if $x=0,$
  (ii) $\langle ax,y \rangle=a \langle x,y \rangle,$  (ii) $\langle x,y \rangle^*=\langle y,x \rangle$ for all $x,y\in \mathcal{E},~~~a\in \mathcal{A}$.
If $\mathcal{E}$ is complete with respect to the norm defined by $\| x\| = \| \langle x, x \rangle\|^\frac{1}{2} $, then $\mathcal{E}$
is called a Hilbert $\mathcal{A}$-module.
 Obviously, every Hilbert space is a Hilbert $\mathbb{C}$-module. Further, every $C^*$-algebra $\mathcal{A}$ can be regarded
 as a Hilbert $C^*$-module over itself under the inner product $\langle a, b \rangle = ab^*$.
 More details on the theory of Hilbert $C^*$-modules can be found, for example, in \cite{La}.

A projection $E \in B(\mathcal{H})$ is minimal if $EB(\mathcal{H})E = \mathbb{C}E $.
 For $\xi, \eta \in \mathcal{H}$, let $\xi \otimes \eta$ denote the corresponding elementary operator
    $(\xi \otimes \eta)(\zeta) = [\zeta, \eta  ] \xi$. It is easy to see that $\|\xi \otimes \eta\| = \| \xi\|\| \eta\|$.
    If $\xi \in \mathcal{H}$ is a unit vector, then $\xi \otimes \xi$ is a minimal projection. Conversely,
    any minimal projection is of this form. The set of all minimal projections in $K(\mathcal{H})$ is denoted by $P_{m}$.
A system $\{x_{\alpha}\}_{\alpha \in \Lambda}$  in a Hilbert $K(\mathcal{H})$-module $\mathcal{E}$ is orthonormal if
(i) $\|x_{\alpha}\|=1$ for all $\alpha \in \Lambda$, (ii)  $\langle x_{\alpha}, x_{\beta} \rangle = 0$ for all $\alpha \neq \beta$,
(iii)  $\langle x_{\alpha}, x_{\alpha} \rangle$ is a minimal projection for all $\alpha \in \Lambda$. An orthonormal system is said to be an orthonormal basis for $\mathcal{E}$ if it generates a dense submodule of $\mathcal{E}$. The orthogonal dimension of $\mathcal{E}$ is defined as the cardinal number of any of its orthonormal bases; see \cite{Bak, CA, MO}.

Let $(X, \| \cdot\|)$ be a complex normed space, and let $x, y \in X$. We say that $x$ is orthogonal to $y$ in the Birkhoff--James sense, denoted by $x \perp_{B} y$, if
$$\| x\| \leq \| x + \alpha y\|\quad \text {for~~~ all }\quad \alpha \in\mathbb{C}.$$
 It is easy to see that the Birkhoff--James orthogonality is equivalent to the usual
orthogonality in the case when $X$ is an inner product space.
 For other definitions and more  properties of the Birkhoff--James orthogonality in Hilbert $C^*$-modules; see \cite{AR1, Bhat,  BG, BO, C, C1, G,T}.

Let $x, y \in X$, we say that $x$ is norm-parallel to $y$, denoted by  $x \| y$, if
$$\| x + \lambda y\| = \| x\| + \| y\| \quad \text{for ~~~some}\quad \lambda \in \mathbb{T} = \{\alpha \in \mathbb{C}: | \alpha| = 1 \}.$$
For more  details one can see \cite{Am,Z1}.

Clearly, two elements of a Hilbert space are norm-parallel if they are linearly dependent. In the case of
normed linear spaces, two linearly dependent vectors are norm-parallel, but the converse is false in general.
Notice that the norm parallelism is symmetric and $\mathbb{R}$-homogenous, but not transitive 
(i.e., $x \|y$ and $y \| z$ do not imply $x \| z$ in general; see \cite[Example 2.7]{Z2},
 unless $X$ is smooth at $y$; see \cite[Theorem 3.1]{W}). Some characterizations of the norm-parallelism for operators and elements of Hilbert $C^*$-modules
were given in \cite{W, Z3, Z1, Z2}.

Some characterizations of the norm-parallelism for elements of a Hilbert $K(\mathcal{H})$-module
are given in Section 2.
 Let $\mathcal{F}$ be a Hilbert $C^*$-module, and let $x, y \in \mathcal{F}$.
 Zamani and Moslehian \cite[Theorem 2.3]{Z2} proved that $x \| y$ if and only if there exist a sequence of unit vectors ${\xi_{n}}$
 in a Hilbert space $\mathcal{H}$ and $\lambda \in \mathbb{T}$ such that 
 
\begin{align*}
\lim_{n \to \infty}{\rm Re}[ \langle x, \lambda y\rangle \xi_{n}, \xi_{n} ]  = \left\| x \right\|\left\| y \right\|.
\end{align*}
Suppose that $\mathcal{E}$ is a Hilbert $K(\mathcal{H})$-module, that $x,y \in  \mathcal{E},$
and that $\langle x, x\rangle$ is a minimal projection in $P_{m}$. We show that $x \| y$  if and only if there exists a unit vector $\xi \in \mathcal{H}$ such that
\begin{align*}
|[ \langle x, y\rangle \xi, \xi ]|  = \left\| y \right\|.
\end{align*}
Baki\'c and Gulja\u s \cite{Bak} proved  that  every Hilbert $K(\mathcal{H})$-module $\mathcal{E}$ has an orthonormal basis.
 Moreover there exists an orthonormal basis $\{x_{\alpha}\}_{\alpha \in \Lambda}$  for $\mathcal{E}$ such that  $\langle x_{\alpha}, x_{\alpha} \rangle = E$
for all $\alpha  \in \Lambda$ and for some minimal projection $E\in P_{m}$; see \cite[Remark 4]{Bak}.
  In the rest of paper, $\{x_{\alpha}\}_{\alpha \in \Lambda}^\xi$ denotes the orthonormal basis of $\mathcal{E}$ such that $\langle x_{\alpha}, x_{\alpha} \rangle = E = \xi \otimes \xi $ for all $\alpha \in \Lambda$ and for some unit vector $\xi \in \mathcal{H}$. We show that there is no element $0 \neq x \in \mathcal{E}$ such that $x \| x_{\alpha}$ for all $\alpha\in \Lambda$.

In Section 3, we use some techniques of \cite{Li1} to characterize matrix pairs which are norm-parallelism with respect to the Schatten $p$-norms.
Let $\| \cdot\|_{\nu}$ be an operator norm on $\mathcal{M}_{n\times n}$ induced by the vector norm $\nu$ on $\mathcal{M}_{n\times n}.$
In \cite{Bhat}, Bhatia, and $\check{S}$emrl conjectured that for $A, B \in \mathcal{M}_{n\times n}$ if $A \perp_{B} B$ in $\| \cdot\|_{\nu}$,
 then $Ay \perp_{B}By$ for some unit vector $y \in \mathcal{M}_{n\times 1}$ with $\nu(y)=1$ such that $\nu(Ay)=\| A\|_{\nu}$.
 Li and Schneider in \cite{Li1} proved that above conclusion is not true in general.

We show that if $Ay\| By $ in $\nu(\cdot)$ for some unit vector $y \in \mathcal{M}_{n\times 1}$ with $\nu(y)=1$ such that $\nu(Ay)=\| A\|_{\nu}$, then $A\| B $  in $\| \cdot\|_{\nu}$, but by \cite[Proposition 4.2 and Example 4.3]{Li1}, the converse is not true in general.

\section{Norm-parallelism in Hilbert $\mathbb{K(\mathcal{H})}$-modules}

Throughout, $\mathcal{E}$ denotes a Hilbert $K(\mathcal{H})$-module. We begin with a characterization of the norm-
parallelism for elements of $\mathcal{E}$.
 Let $\mathcal{F}$ be a Hilbert $C^*$-module and let $x, y \in \mathcal{F}$. Zamani and Moslehian \cite[Theorem 2.3]{Z2} proved that $x \| y$
 if and only if there exist a sequence of unit vectors ${\xi_{n}}$
 in a Hilbert space $\mathcal{H}$ and $\lambda \in \mathbb{T}$ such that
\begin{align*}
\lim_{n \to \infty}{\rm Re}[ \langle x, \lambda y\rangle \xi_{n}, \xi_{n} ]  = \left\| x \right\|\left\| y \right\|.
\end{align*}
Now suppose that $x, y \in \mathcal{E}$.
Following we show that $x \| y$ if and only if there exists a unit vector $\xi \in \mathcal{H}$ such that
$ |[ \langle x, x \rangle   \langle x, y \rangle \xi, \xi ]| =\| x\|^3 \| y\| $.
\begin{theorem}\label{a}
Let $x, y \in \mathcal{E}$. Then the following statements are equivalent:
\begin{enumerate}
\item[(i)] $x \| y$.
\item[(ii)] There exists a unit vector $\xi \in \mathcal{H}$ such that
\begin{align*}
 |[ \langle x, x \rangle   \langle x, y \rangle \xi, \xi ]| =\| x\|^3 \| y\|.
\end{align*}
\end{enumerate}
\end{theorem}
\begin{proof}
$(i) \Rightarrow (ii)$: Let $x \| y$. We may assume that $x \neq 0$. From \cite[Theorem 4.7]{Z1} $\langle x, x\rangle \| \langle x, y\rangle $ and $\| \langle x, y \rangle\|  = \| x\| \| y\|$.
Theorem 2.4 of \cite{Z2} states that there exists $\lambda \in \mathbb{T}$ such that
\begin{align*}
\langle x, x\rangle  \perp_{B} \|\langle x, y\rangle\| \langle x, x\rangle + \lambda \| x\|^2 \langle x, y\rangle.
\end{align*}
By the proof of \cite[Theorem 2.10]{Z3}, there exists a unit vector $\xi \in \mathcal{H}$ such that $\|  \langle x, x\rangle\xi\| = \|  \langle x, x\rangle\|$ and
\begin{align*}
\bigg[ \Big(\|\langle x, y\rangle\| \langle x, x\rangle + \lambda \| x\|^2 \langle x, y\rangle \Big)\xi, \langle x, x\rangle \xi  \bigg] = 0.
\end{align*}
Thus $ | [ \langle x, x \rangle   \langle x, y \rangle \xi, \xi ] | =\| x\|^2 \| \langle x, y \rangle\|=\| x\|^3 \| y\|$.

$(ii) \Rightarrow(i)$: Suppose that $(ii)$ holds.
Then there exists a unit vector $\xi \in \mathcal{H}$ such that
\begin{align*}
\| x\|^3 \| y\| = | [ \langle x, x \rangle   \langle x, y \rangle \xi, \xi ] | &\leq \| x\| ^2\|\langle x, y\rangle\| \leq \| x\|^2\| x\| \| y\|.
\end{align*}
Hence $\| \langle x, y \rangle\|  = \| x\| \| y\|$ and
\begin{align*}
 |[ \langle x, x \rangle   \langle x, y \rangle \xi, \xi ]| = \| x\|^2 \| \langle x, y \rangle\|=\| \langle x, x \rangle\| \| \langle x, y \rangle\|.
\end{align*}
By \cite[Theorem 2.10]{Z3}, we have $\langle x, x\rangle  \| \langle x, y\rangle $.
Then \cite[Theorem 4.7]{Z1} implies that  $x \| y$.
\end{proof}

\begin{lemma}\cite[Corollary 2.5]{Z2}\label{l1}
Let $x, y \in \mathcal{E}$. If $x \| y$, then there exists $\lambda \in \mathbb{T}$ such that

 $x \perp_{B} (\| y\| \langle x, x \rangle x + \lambda\| x\| \langle y, x \rangle x)$   and   $y \perp_{B} (\|x\| \langle y, y \rangle y + \lambda\| y\|    \langle y, x \rangle y)$.
\end{lemma}
Now, we apply Lemma \ref{l1} to obtain the following characterization of the norm-parallelism for elements of a Hilbert $K(\mathcal{H})$-module.

\begin{corollary}
Let $x, y \in \mathcal{E}$ and $ \langle x, x\rangle $ be idempotent. Then the following statements are equivalent:
\begin{enumerate}
\item[(i)] $x \| y$;
\item[(ii)] $x \perp_{B} (\| y\| \langle x, x \rangle x + \lambda\| x\| \langle y, x \rangle x)$   and   $y \perp_{B} (\|x\| \langle y, y \rangle y + \lambda\| y\|    \langle y, x \rangle y)$
\end{enumerate}
for some $\lambda \in \mathbb{T}$.
\begin{proof}
$(i) \Rightarrow(ii)$: This implication follows immediately from Lemma \ref{l1}.

$(ii) \Rightarrow(i)$: Suppose that $(ii)$ holds. By \cite[Lemma 2.1 part(5)
]{AR1}, we have
\begin{align*}
\langle x, x \rangle \perp_{B} \Big(\| y\| \langle x, x \rangle^2 + \bar{\lambda}\| x\|  \langle x, x \rangle\langle x, y \rangle \Big).
\end{align*}
By \cite[Theorem 2.10]{Z3}, there exists a vector $\xi$ in  $\mathcal{H}$ such that
\begin{align*}
\bigg[ \Big(\| y\| \langle x, x \rangle^2 + \bar{\lambda}\| x\| \langle x, x\rangle\langle x, y \rangle\Big)\xi ,\langle x, x \rangle\xi\bigg] = 0
\end{align*}
and
\begin{align*}
\| \langle x, x \rangle\xi\| = \| x\|^2.
\end{align*}
 Since $\langle x, x \rangle$ is idempotent, we reach
\begin{align*}
|[ \langle x, x \rangle   \langle x, y \rangle \xi, \xi ]| = \| x\|^3 \| y\|.
\end{align*}
From $(ii) \Rightarrow(i)$ of Theorem \ref{a}, we obtain $x \| y$.
\end{proof}
\end{corollary}
Recall that if $x$ is an element of $\mathcal{E}$ such that $\langle x, x \rangle$ is a minimal projection, 
 then $\langle x, x \rangle x = x$; see \cite{Bak}.
 Let $ x, y\in \mathcal{E}$.  If there exist a unit vector $\xi \in \mathcal{H}$ and $\lambda \in \mathbb{T}$ such that 
 ${\rm Re} [ \langle x,\lambda y \rangle \xi, \xi ] =\| x\| \| y\|$,  then \cite[Theorem 2.3]{Z2}
 implies that $x \| y$. The question is under which conditions the converse is true. When $\langle x, x \rangle$ is a minimal projection, by using Theorem \ref{a}, we get $x \| y$, if and only if there exists a unit vector $\xi\in \mathcal{H}$ such that $\Big| [ \langle x, y  \rangle \xi, \xi ] \Big| = \| y\|.$ Now we prove it by a different approach. 

\begin{theorem} \label{L}
Let $ x, y\in \mathcal{E}$. If $\langle x, x \rangle$ is a minimal projection, then $ x \| y$  if and only if there exists a unit vector $\xi\in \mathcal{H}$ such that
\begin{align*}
 \Big| [ \langle x, y  \rangle \xi, \xi ] \Big| = \| y\|.
\end{align*}
\begin{proof}
Let $x \| y$. By Lemma \ref{l1}, there exists a scalar $\mu$ in $\mathbb{T}$ such that
\begin{align*}
x \perp_{B} \Big(\| y\| \langle x, x \rangle x + \mu\| x\| \langle y, x \rangle x\Big).
\end{align*}
Since $\langle x, x \rangle$ is a minimal projection, by \cite[Proposition 2.4]{AR1}, we have
\begin{align*}
\langle x, x \rangle\Big\langle  \Big(\| y\| \langle x, x \rangle x + \mu   \langle y, x \rangle x\Big), x \Big\rangle = 0,
\end{align*}
and so
\begin{align*}
\| y\| \langle x, x \rangle + \mu \langle x, x \rangle  \langle y, x \rangle \langle x, x \rangle   = 0.
\end{align*}
Hence
\begin{align}
\langle x, x \rangle  \langle x, y \rangle \langle x, x \rangle   = - \bar{\mu}  \| y\| \langle x, x \rangle.\label{0}
\end{align}
Now, by knowing the fact that $\langle x, x \rangle = \xi \otimes \xi$ for some unit vector $\xi \in \mathcal{H}$, since $\xi \otimes \xi  \langle x, y \rangle \xi \otimes \xi =  [\langle x, y \rangle \xi, \xi ]\xi \otimes \xi,$  we obtain
\begin{align*}
\Big|[ \langle x, y \rangle \xi, \xi ]\Big|\| \xi \otimes \xi\| = \| y\| \| \xi \otimes \xi\|,
\end{align*}
whence
\begin{align*}
\Big|[ \langle x, y \rangle \xi, \xi ]\Big| = \| y\|.
\end{align*}
$(ii) \Rightarrow(i)$: Let $\Big|[ \langle x, y \rangle \xi, \xi ]\Big| = \| y\|.$
Then there exists $\lambda \in \mathbb{T}$ such that
\begin{align*}
{\rm Re}[ \langle x, \lambda y\rangle \xi, \xi]  = \| x\|\left\| y \right\|.
\end{align*}
 From  \cite[Theorem 2.3]{Z2}, we get $x\| y$.
\end{proof}
\end{theorem}
In the following theorem we use some ideas of \cite{Z2}.

\begin{theorem}\label{b}
Let $\{x_{\alpha}\}_{\alpha \in \Lambda}^\xi$ be an orthonormal basis for $\mathcal{E}$ 
such that as a Hilbert module is not of dimension one.
 Then there is no nonzero element $x \in \mathcal{E}$ such that $x \| x_{\alpha}$ for all $\alpha\in \Lambda$.
\begin{proof}
 Let there be a nonzero element $ x \in \mathcal{E}$ such that $x \| x_{\alpha}$ for all $\alpha\in \Lambda$. By \cite[Corollary 2.5]{Z2}, there exists a sequence $\{\lambda_\alpha\}$ in $\mathbb{T}$ such that
\begin{align}
x_{\alpha} \perp_{B} \Big(\|x\| \langle x_{\alpha}, x_{\alpha} \rangle x_{\alpha} + \lambda_\alpha   \langle x_{\alpha}, x \rangle x_{\alpha}\Big)\,\,\,\,\,(\alpha\in \Lambda).\label{1}
\end{align}
 Taking $E = \langle x_{\alpha}, x_{\alpha} \rangle$, since $E x_\alpha = E,$ equation \eqref{0} implies that
\begin{align*}
\langle x_{\alpha}, x \rangle  E = E \langle x_{\alpha}, x \rangle  E= -\overline{\lambda_\alpha}\|x\| E\,\,\,\,\,(\alpha\in \Lambda).
\end{align*}
Hence
\begin{align}
(\langle x_{\alpha}, x \rangle  E)^* (\langle x_{\alpha}, x \rangle  E) = \lambda_\alpha \overline{\lambda_\alpha}\| x\|^2 E^2 = \| x\|^2 E\label{3}
\end{align}
or equivalently 
\begin{align*}
E\langle x, x_{\alpha} \rangle\langle x_{\alpha}, x \rangle  E = \| x\|^2 E\,\,\,\,\,(\alpha\in \Lambda).
\end{align*}
Therefore $[ \langle x, x_{\alpha} \rangle \langle x_{\alpha}, x \rangle \xi, \xi ] = \| x\|^2$
and
\begin{align*}
[ \langle x, x  \rangle \xi, \xi  ] = \Big[\sum_{\alpha\in \Lambda} \langle x, x_{\alpha}  \rangle  \langle x_{\alpha}, x  \rangle\xi, \xi  \Big] = \sum_{\alpha\in \Lambda}\|x\| ^2.
\end{align*}
Thus
\begin{align*}
\sum_{\alpha\in \Lambda}\|x\| ^2 = [ \langle x, x  \rangle \xi, \xi ] = \Big|[ \langle x, x  \rangle \xi, \xi ]\Big| \leq \| [ \langle x, x  \rangle \xi, \xi ]\| \leq \| x\|^2,
\end{align*}
which implies that $x =0,$ that is a contradiction.
\end{proof}
\end{theorem}
Now we have the following characterization of the transitive relation for elements of a certain Hilbert $K(\mathcal{H})$-module.

\begin{corollary}\label{transitive}
Let $x, y, z\in \mathcal{E}$ be such that $\mathcal{E}$ as a Hilbert module is of dimension one with orthonormal basis $\{y\}^\xi$. If  $x \| y$ and $y\| z$, then $x\| z$.
\begin{proof}
Let $x \| y$ and $y\| z$. By \cite[Corollary 2.5]{Z2}, there exist $\lambda,\,\, \mu \in\mathbb{T}$ such that $y \perp_{B} (\|x\| \langle y, y \rangle y + \lambda\| y\|    \langle y, x \rangle y)$ and $y \perp_{B} (\| z\| \langle y, y \rangle y + \mu\| y\| \langle z, y \rangle y).$ Taking $E = \langle y, y \rangle$.
Equation (\ref{0}) ensures that $E \langle x, y \rangle = -\bar{\lambda}\|x\| E$ and $\langle y, z \rangle  E= -\bar{\mu}\|z\| E$.
Then $E\langle x, y \rangle \langle y, z \rangle  E = \bar{\lambda}\bar{\mu}\left\|x \right\| \left\|z \right\| E$. Therefore
\begin{align*}
\Big|[ \langle x, z  \rangle \xi, \xi ] \Big|= \Big|[  \langle x, y  \rangle  \langle y, z  \rangle\xi, \xi ]\Big| = \left\| x \right\| \left\| z \right\|.
\end{align*}
That is  $x\Arrowvert z$.
\end{proof}
\end{corollary}

\begin{example}
Suppose that $\mathcal{H}$ is an infinite dimensional Hilbert space. Then $\mathcal{H}$ equipped with
 $\langle x, y \rangle = x \otimes y $ is a Hilbert $K(\mathcal{H})$-module and
 $\|x\|^2=\|x \otimes x\|$ is the original norm on $\mathcal{H}$.
We can assume that each unit vector $y \in \mathcal{H}$ gives an orthonormal basis $\{y\}^y$ for $\mathcal{H}$.
Therefore the dimension of $H$ as a Hilbert module is equal to one.
Finally, from Corollary \ref{transitive}, we see that if $x \| y $ and $y \| z$, for some unit vector $y$, then
\begin{align*}
\Big|[ \langle x, z  \rangle y, y ]\Big|  = \| x\| \| z\|.
\end{align*}
Therefore, $x \|z $.
\end{example}

\section{Norm-parallelism in matrices}\label{s3}

Let $\mathcal{M}_{n\times n}$ be the space of $n \times n$ complex matrices, let $A, B \in \mathcal{M}_{n\times n}$,
 and let $\| \cdot\|$ be any norm on $\mathcal{M}_{n\times n}$.  We say that $A$ is norm-parallel to $B$ if
\begin{align*}
 \| A + \lambda B\| = \| A\| + \| B\|
\end{align*}
for some $\lambda \in \mathbb{T}$.
We also consider the inner product $\langle A, B  \rangle ={\rm tr}(AB^*)$.
 The dual norm of $\| \cdot\|$ is defined by
\begin{align*}
\| A\|^D = \max\{|\langle A, B  \rangle|: \| B\|\ \leq 1\}.
\end{align*}
 The singular values of $A \in \mathcal{M}_{n\times n}$, which are the  square roots of the eigenvalues of the matrix $A^*A$,
 are denoted by $s_{1}(A)\geq \cdots  \geq s_{n}(A) $ and  the Schatten $p$-norm of  $A$ is defined by
 \begin{align*}
\|A\|_p = \bigg\{\sum_{i = 1}^ns_{i}(A)^p\bigg\}^\frac{1}{p},
\end{align*}
if $1\leq p< \infty$ and if $p = \infty,~~~\|A\|_p = \max \{s_{j}(A):~~~ 1\leq j \leq n \}$; see  \cite{Li1}.
If $A \in \mathcal{M}_{n\times n}$, then for a positive integer $1\leq k\leq n$, the Ky-Fan $k$-norm is defined by
\begin{align*}
\| A\|_{(k)} = \sum_{j = 1}^ks_{j}(A),
\end{align*}
where $s_{1}(A)\geq \cdots  \geq s_{n}(A) \geq 0 $, are the singular values of $A$. The cases when $k = 1$ and $k = n$ correspond to the operator norm  $\| \cdot\|_{\infty}$ and the trace norm $\| \cdot\|_{1}$, respectively. we have the following result.

 \begin{proposition}\label{p2}
 Let $A, B \in \mathcal{M}_{n\times n}$. Then the  following statement are equivalent:

 \begin{enumerate}
\item[(i)] $ A \| B $, in $\|\cdot\|_{\infty}$.

\item[(ii)] There exist unit vectors  $x \in \mathcal{M}_{n\times 1}$ and $y \in \mathcal{M}_{n\times 1}$ such that $\|A\|_{\infty}=x^*Ay$ and
$| x^*By|= \|B\|_{\infty}$.
\item[(iii)] For any $U \in \mathcal{M}_{n\times n} $ with orthonormal columns that form a basis for the eigenspace of $AA^*$
corresponding to the largest eigenvalue and $V = \frac{1}{\|A\|_{\infty}}A^*U\in  \mathcal{M}_{n\times n}$, it holds that
 \begin{align*}
-\|B\|_{\infty}\in W(\lambda U^*BV)
\end{align*}
for some $\lambda \in \mathbb{T}$.
\end{enumerate}
\end{proposition}
 \begin{proof}
$(i) \Rightarrow (ii)$: Let $ A \| B $. Thus
 \begin{align*}
A \perp_{B} (\|B\|_{\infty}A + \lambda \|A\|_{\infty}B)
\end{align*}
for some $\lambda \in \mathbb{T}$. By \cite[Theorem 3.1]{Li1}, there exist unit vectors  $x \in \mathcal{M}_{n\times 1}$ and $y \in \mathcal{M}_{n\times 1}$ such that $\|A\|_{\infty}=x^*Ay$ and $x^*(\|B\|_{\infty}A + \lambda \|A\|_{\infty}B)y= 0$. Thus we get  $| x^*By|= \|B\|_{\infty}$.

$(ii) \Rightarrow (i)$: It is obvious.

$(i) \Rightarrow (iii)$: If $ A \| B $, then
 \begin{align*}
A \perp_{B} (\|B\|_{\infty}A + \lambda \|A\|_{\infty}B)
\end{align*}
for some $\lambda \in \mathbb{T}$. By \cite[Theorem 3.1]{Li1}, for any $U \in \mathcal{M}_{n\times n}$
 with orthonormal columns that form a basis for the eigenspace of $AA^*$ corresponding to the largest eigenvalue and $V = \frac{1}{\|A\|_{\infty}}A^*U\in  \mathcal{M}_{n\times n}$, we have
 \begin{align*}
0 \in W( U^*(\|B\|_{\infty}A + \lambda \|A\|_{\infty}B)V) &= W(\|B\|_{\infty}\|A\|_{\infty}V^*V + \lambda \|A\|_{\infty}U^*BV)\\
&=W(\|B\|_{\infty} \|A\|_{\infty}I + \lambda \|A\|_{\infty}U^*BV).
\end{align*}
Thus there exists a unit vector $\zeta \in \mathcal{M}_{n\times 1}$  such that $\langle (\|B\|_{\infty}I + \lambda \|A\|_{\infty}U^*BV)\zeta, \zeta \rangle=0$
or\\
 $\langle \lambda U^*BV \zeta, \zeta \rangle=-\frac{\|B\|_{\infty}}{\|A\|_{\infty}}$.
Hence
 \begin{align*}
-\|B\|_{\infty} \in W(\lambda U^*BV)
\end{align*}
for some $\lambda \in \mathbb{T}$.

$(iii) \Rightarrow (i)$: It is obvious.
 \end{proof}

 \begin{remark}
By the equivalence (i)$\Leftrightarrow$(ii) of Proposition \ref{p2}, $ A \| B $, in $\|\cdot\|_{\infty}$ if and only if there exist unit vectors  $x \in \mathcal{M}_{m\times 1}$ and $y \in \mathcal{M}_{n\times 1}$ such that $\|A\|_{\infty}=x^*Ay$ and $| x^*By|= \|B\|_{\infty}$, this condition is equivalent to
the fact that there is a unit vector $y \in \mathcal{M}_{n\times 1}$ such that $\|A\|_{\infty} = \langle Ay, Ay  \rangle^\frac{1}{2}$ and $|\langle Ay, By  \rangle| = \|A\|_{\infty}\|B\|_{\infty}$.
 \end{remark}
\begin{remark}
Suppose that $A, B \in \mathcal{M}_{n\times n}$. If $A \| B$ in $\| \cdot\|_{(k)}$, then $A \perp_{B} (\| B\|_{(k)}A + \lambda \| A\|_{(k)}B)$.
So,  by \cite[Remark 1]{Gr}, there exists a matrix  $F \in \mathcal{M}_{n\times n}$ such that $\| F\|_{\infty}\leq 1,~~~\| F\|_{1}\leq k,~~~{\rm tr}(F^*A) = \| A\|_{(k)}$ and ${\rm tr}\bigg(F^*(\| B\|_{(k)}A + \lambda \| A\|_{(k)}B)\bigg) = 0$. Hence $A \| B$ in $\| \cdot\|_{(k)}$ if and only if there exists a matrix  $F \in \mathcal{M}_{n\times n}$ such that $\| F\|_{\infty}\leq 1,~~~\| F\|_{1}\leq k,~~~{\rm tr}(F^*A) = \| A\|_{(k)}$ and $|{\rm tr}(F^*B)| = \| A\|_{(k)}$.
\end{remark}

Now by \cite[Theoren 3.2]{Li1}, we obtain the following result.

\begin{corollary}
  Let $1<p<\infty$ and $A, B \in \mathcal{M}_{n\times n}$, where $A=DC$ for some positive matrices
   $D \in \mathcal{M}_{m\times m}$ and $C \in \mathcal{M}_{m\times n}$ with $CC^*=I_{n}$.
   Then $ A \| B $, in $\|\cdot\|_p$ if and only if  $|{\rm tr}(D^{p-1}CB^*)| = \frac{\|B\|_p\|D\|_p^p}{\|A\|_p}$.
    \end{corollary}
 \begin{proof}
 Let $ A \| B $. Thus
 \begin{align*}
A \perp_{B} (\|B\|_pA + \lambda \|A\|_pB)
\end{align*}
for some $\lambda \in \mathbb{T}$. By \cite[Theorem 3.2]{Li1}, ${\rm tr}(D^{p-1}C(\|B\|_pA + \lambda \|A\|_pB)^*)=0$.\\
 Thus $|{\rm tr}(D^{p-1}CB^*)| = \frac{\|B\|_p\|D\|_p^p}{\|A\|_p}$.
The converse is obvious.
 \end{proof}

 As a consequence of \cite[Theorem 3.3]{Li1}, we have the following result.

\begin{remark}
 Let $A, B \in \mathcal{M}_{n\times n}$ and $A \| B$ in $\| \cdot \|_1$.
 By a similar computation as previous theorem, we deduce that
 there exists $F \in \mathcal{M}_{n\times n}$ such that $\|F\|_{\infty}\leq 1,~~~{\rm tr}(AF^*)=\|A\|_1$ and $|{\rm tr}(BF^*)| = \|B\|_1$.
 \end{remark}

Let $\nu$ be a norm on $\mathcal{M}_{n\times 1}$, and let $\| \cdot\|_{\nu}$ be the operator norm on $\mathcal{M}_{n\times n}$ induced by $\nu$
defined by
\begin{align*}
\| A\|_{\nu} = \max\{\nu(Ax): x \in \mathcal{M}_{n\times 1}, \nu(x)\leq 1\}.
\end{align*}
The dual norm of $\nu$ is defined by
\begin{align*}
\nu^D(x) = \max\{|\langle x, y  \rangle|: y \in \mathcal{M}_{n\times 1}, \nu(y)\leq 1\}
\end{align*}
and the dual norm of  $\| \cdot\|_{\nu}$ is defined as
\begin{align*}
\| A\|_{\nu}^D = \max\{|\langle A, B  \rangle|: B \in \mathcal{M}_{n\times n}, \|B\|_{\nu}\leq 1\}.
\end{align*}

We say $F \in \mathcal{M}_{n\times n}$ is an extreme point of the unit norm ball of $\mathcal{M}_{n\times n}$ if only if

\begin{enumerate}
\item[(i)] $p = 1$ and $F= xy^*$ for some unit vectors $ x \in \mathcal{M}_{n\times 1}$ and $ y \in \mathcal{N}_{n\times 1}$,

\item[(ii)] $1<p<\infty$ and $\|F\|_{p}= 1$,

\item[(iii)] $p = \infty$ and $FF^*=I_{n}$.
\end{enumerate}

We have the following characterization of the norm parallelism for operator norms in $\mathcal{M}_{n\times n}$.

\begin{proposition}
Let $\| \cdot\|_{\nu}$ be an operator norm on $\mathcal{M}_{n\times n}$ induced by the vector norm $\nu$ on $\mathcal{M}_{n\times n}$.
Let $A, B \in \mathcal{M}_{n\times n}$. Denote by $\mathscr{E}$ and $\mathscr{E}^D$ the set of extreme points of the unit norm balls of $\nu$ and $\nu^D$, respectively. If
\begin{align*}
V(A) = \{xy^*: x\in \mathscr{E}^D, y\in \mathscr{E}, \langle A, xy^*\rangle = \| A\|_{\nu}\},
\end{align*}
then $A \| B$ in $\| \cdot\|_{\nu}$ if and only if there exist $k$ extreme points $x_{1}y_{1}^*, \ldots,x_{k}y_{k}^*$ with $k\leq 3 $ in the complex case and $k\leq 2 $ in the real case, and positive numbers $t_{1}, \ldots, t_{k}$ with $t_{1} + \cdots + t_{k} = 1$ such that
\begin{align*}
|\sum_{j=1}^k t_{j}\langle B, x_{j}y_{j}^*\rangle |= \| B\|_{\nu}.
\end{align*}
\end{proposition}
\begin{proof}
 Let $ A \| B $. Then
 \begin{align*}
A \perp_{B} (\| B\|_{\nu}A + \lambda \| A\|_{\nu}B)
\end{align*}
for some $\lambda \in \mathbb{T}$. So by  \cite[Proposation 4.2]{Li1},  there exist $k$ extreme points $x_{1}y_{1}^*, \ldots,x_{k}y_{k}^*$ with $k\leq 3 $ in the complex case and $k\leq 2 $ in the real case, and positive numbers $t_{1}, \ldots, t_{k}$ with $t_{1} + \cdots + t_{k} = 1$ such that
 \begin{align*}
\sum_{j=1}^k t_{j}\langle (\| B\|_{\nu}A + \lambda \| A\|_{\nu}B), x_{j}y_{j}^* \rangle = 0.
\end{align*}
Thus $|\sum_{j=1}^k t_{j}\langle B, x_{j}y_{j}^* \rangle|= \| B\|_{\nu}$. The converse  is obvious.
\end{proof}

\begin{remark}\label{R2}
Suppose that $\nu$ is a norm on $\mathcal{M}_{n\times 1}$ and  that $\| \cdot\|_{\nu}$ is the corresponding operator norm on  $\mathcal{M}_{n\times n}$. Given $A, B \in \mathcal{M}_{n\times n}$. If there exists a vector $y \in \mathcal{M}_{n\times 1}$ with $\nu(y)=1$ such that $\nu(Ay)=\| A\|_{\nu}$ and $\nu\bigg(Ay + \mu(\nu(By)Ay + \lambda\nu(Ay)By)\bigg) \geq \nu(Ay)$ for all $\mu \in \mathbb{C}$, that is, $ Ay \| By $. Then
 \begin{align*}
\| A +\mu(\| B\|_{\nu}A + \lambda \| A\|_{\nu}B)\|_{\nu}\geq \nu\bigg(Ay + \mu(\nu(By)Ay + \lambda\nu(Ay)By)\bigg) \geq \nu(Ay) = \| A\|_{\nu}
\end{align*}
for all  for all $\mu \in \mathbb{C}$; hence $A\|B$. On the other hand  \cite[Proposation 4.2 and Example 4.3]{Li1} show that  the converse of the above result is not true, in general.
\end{remark}


\end{document}